\documentclass[12pt]{article}
\usepackage{amsfonts}
\usepackage{amsmath,amscd,amsthm,a4,amssymb}
\usepackage{amsmath,amscd,amsthm,a4,amssymb}

\newtheorem{theorem}{Theorem}
{}
\newtheorem{corollary}{Corollary}
{}
\newtheorem{definition}{Definition}
{}
\newtheorem{remark}{Remark}
{}

{}
\theoremstyle{plain}
{}

{}

{}

{}

{}

{}

{}

{}

{}

{}

{}

{}

{}

{}

{}

{}

\begin{document}
	\begin{center}
		{\Large \bf{Quaternionic $\left( 1,3\right) -$%
				Bertrand Curves According to Type 2-Quaternionic Frame in $\mathbb{R}^{4}$ }}
	\end{center}
	\centerline{\large Ferdag KAHRAMAN AKSOYAK  {\footnotetext{
				{E-mail: ferdag.aksoyak@ahievran.edu.tr(F. Kahraman Aksoyak) }} }}
	
	\centerline{\it Ahi Evran University, Division of Elementary Mathematics Education
		Kirsehir, Turkey}

\begin{abstract}
	If there exists a quaternionic Bertrand curve in $\mathbb{E}^{4}$, then its
	torsion or bitorsion vanishes. So we can say that there is no quaternionic
	Bertrand curves whose torsion and bitorsion are non-zero. Hence by using the
	method which is given by Matsuda and Yorozu \cite{mat},\ we give the
	definition of quaternionic $(1,3)-$Bertrand curve according to Type
	2-Quaternionic Frame and obtain some results about these curves..
\end{abstract}

\begin{quote}\small
\end{quote}
\begin{quote}\small
\end{quote}

\section{Introduction}

Bertrand curve was introduced by \textit{Bertrand }in 1850 (see \cite{bert}%
). When a curve is given, if there exists a second curve whose the principal
normal is the principal normal of that curve, then the first curve is called
Bertrand curve and the second curve is called the Bertrand mate of the first
curve. The most important properties of Bertrand curves in Euclidean 3-space
are that the distance between corresponding points is constant and there is
a linear relation between the curvature functions of the first curve, that
is, for $\lambda ,\mu $ $\in \mathbb{R}$, $\lambda \kappa +\mu \tau =1$,
where $\kappa $ is curvature and $\tau $ is the torsion of the first curve.
Also the absolute value of the real number $\lambda $ in this linear
relation is equal to the distance between corresponding points of Bertrand
curves. The Bertrand curves in Euclidean 3-space was extended by L. R. Pears
to Riemannian $n-$space and gave general results for Bertrand curves \cite%
{Pears}. If these general results were applied to Euclidean $n-$space, then
either torsion or bitorsion of the curve vanishes. Otherwise, for $n\geq 4,$
then no special Frenet curve in $\mathbb{E}^{n}$ is a Bertrand curve \cite%
{mat}. Hence, Matsuda and Yorozu gave a new definition of Bertrand curve
which is called $(1,3)-$Bertrand curve and obtain a characterization of $%
(1,3)-$Bertrand curve. \cite{mat}. After then many researchers have made a
lot of papers about $(1,3)-$Bertrand curves \cite{er},\  \cite{kaz4}, \cite%
{kaz5}, \cite{kaz1}, \cite{kaz2}, \cite{kaz3}.

In 1987, Bharathi and Nagaraj introduced the Serret-Frenet formulas for
spatial quaternionic curves in $\mathbb{R}^{3}$ and quaternionic curves in $%
\mathbb{R}^{4}$ \cite{nag}. Since the quaternionic multiplication of two
orthogonal vectors in $\mathbb{R}^{3}$ becomes vector product of these
vectors, they reconsider the Serret- Frenet formulae of any curve in $%
\mathbb{R}^{3}$ which is well known in differential geometry by using the
quaternionic multiplication and then they compose the Serret- Frenet
formulae of quaternionic curves in $\mathbb{R}^{4}$ by means of the the
Serret-Frenet formulas of spatial quaternionic curves in $\mathbb{R}^{3}$ 
\cite{nag}. After then various studies have been carried out on the
adaptation of some special curves to quaternionic curves \cite{koca},\  \cite%
{gok1}, \cite{gun}, \cite{mak}, \cite{kara}, \cite{on}, \cite{gunay}, \cite%
{yil}, \cite{yoon1}, \cite{yoon}. Ke\c{c}ilio\u{g}lu and \.{I}larslan
defined $(1,3)-$Bertrand curves for quaternionic curves in Euclidean 4-space
and obtained a characterization for such curves \cite{kaz}.

Also, Kahraman Aksoyak defined a new type of quaternionic frame for
quaternionic curves in Euclidean 4- space which is called Type
2-Quaternionic Frame \cite{ak}.

In this paper, by using the method which is given by Matsuda and Yorozu \cite%
{mat},\ we give the definition of quaternionic $(1,3)-$Bertrand curve
according to Type 2-Quaternionic Frame and obtain some results about these
curves.

\section{Preliminaries}

A real quaternion is defined as: 
\begin{equation*}
q=q_{0}+q_{1}e_{1}+q_{2}e_{2}+q_{3}e_{3}
\end{equation*}%
where $q_{t}\in \mathbb{R}$ for $0\leq t\leq 3$ and $e_{1},$ $e_{2},$ $e_{3}$
are unit vectors in usual three dimensional real vector space. Any
quaternion $q$\ can be divided into two parts such that the scalar part
denoted by $S_{q}$ and the vectorial part denoted by $V_{q}$, where $%
S_{q}=q_{0}$ and $V_{q}=q_{1}e_{1}+q_{2}e_{2}+q_{3}e_{3}$. So, we can
rewrite any real quaternion as $q=S_{q}+V_{q}$. If $q=S_{q}+V_{q}$ and $%
q^{\prime }=S_{q^{\prime }}+V_{q^{\prime }}$ are any two quaternions,
addition, the multiplication by a real scalar $c$\ and the conjugate of $q$
denoted by $\gamma q$ are defined as, respectively: 
\begin{eqnarray*}
	q+q^{\prime } &=&\left( S_{q}+S_{q^{\prime }}\right) +\left(
	V_{q}+V_{q^{\prime }}\right) \\
	cq &=&cS_{q}+cV_{q} \\
	\gamma q &=&S_{q}-V_{q}
\end{eqnarray*}%
Let denote the set of quaternions by $Q.$ $Q$ is a real vector space
according to this addition and scalar multiplication. A basis of this vector
space is $\left \{ 1,\text{ }e_{1},\text{ }e_{2},\text{ }e_{3}\right \} $
and it is a four dimensional vector space. Hence we can think of any
quaternion $q$ as an element $\left( q_{0},q_{1},q_{2},q_{3}\right) $ of $%
\mathbb{R}^{4}. $ Even a quaternion whose the scalar part is zero (it is
called spatial quaternion) can be considered as a ordered triple $\left(
q_{1},q_{2},q_{3}\right) $ of $\mathbb{R}^{3}.$

The product of two quaternions is defined by means of the multiplication
rule between the units $e_{1},$ $e_{2},$ $e_{3}$ are given by: 
\begin{equation}
e_{1}^{2}=e_{2}^{2}=e_{3}^{2}=e_{1}e_{2}e_{3}=-1
\end{equation}

So, by using (1), quaternionic multiplication is obtained as:%
\begin{equation*}
q\times q^{\prime }=S_{q}S_{q^{\prime }}-\langle V_{q},V_{q^{\prime
}}\rangle +S_{q}V_{q^{\prime }}+S_{q^{\prime }}V_{q}+V_{q}\wedge
V_{q^{\prime }}\text{ for every }q,\text{ }q^{\prime }\in Q,
\end{equation*}%
where $\left \langle ,\right \rangle $ and $\wedge $ denote the inner
product and cross products in $\mathbb{R}^{3},$ respectively$.$ The
quaternion multiplication is associative and distributed but
non-commutative. So $Q$ is a real algebra and it is called quaternion
algebra.

Now, the symetric, non-degenerate, bilinear form $h$ on $Q$ is defined as :%
\begin{equation*}
h:Q\times Q\rightarrow \mathbb{R}\text{,}
\end{equation*}%
\begin{equation*}
h(q,q^{\prime })=\frac{1}{2}(q\times \gamma q^{\prime }+q^{\prime }\times
\gamma q)\text{ for }q,\text{ }q^{\prime }\in Q
\end{equation*}%
and the norm of any real quaternion $q$ is determined as: 
\begin{equation*}
\left \Vert q\right \Vert ^{2}=h(q,q)=q\times \gamma q=S_{q}^{2}+\left
\langle V_{q},V_{q}\right \rangle .
\end{equation*}%
So the mapping $h$ is called the quaternion (or Euclidean) inner product 
\cite{nag}.

In this paper, a quaternionic curve in $\mathbb{R}^{4}$ is denoted by $%
\alpha ^{\left( 4\right) }$\ and the spatial quaternionic curve in $\mathbb{R%
}^{3}$ associated with $\alpha ^{\left( 4\right) }$ in $\mathbb{R}^{4}$ is
denoted by $\alpha .$ Type 2-Quaternionic Frame for a quaternionic curve in $%
\mathbb{R}^{4}$ is defined as\cite{ak}:

\begin{theorem}
	\label{teo 2.1}Let $I=\left[ 0,1\right] $ denote the unit interval in the
	real line $\mathbb{R}$ and $S$ be the set of spatial quaternionic curve%
	\begin{equation*}
	\alpha :I\subset \mathbb{R}\longrightarrow S,
	\end{equation*}%
	\begin{equation*}
	\text{ \  \  \  \  \  \  \  \  \  \  \  \  \  \  \  \  \  \  \  \  \  \  \  \  \  \  \  \ }%
	s\longrightarrow \alpha (s)=\alpha _{1}(s)e_{1}+\alpha _{2}(s)e_{2}+\alpha
	_{3}(s)e_{3}
	\end{equation*}%
	be an arc-lenghted curve. Then the Frenet equations of $\alpha $\ are as
	follows:%
	\begin{equation*}
	\left[ 
	\begin{array}{c}
	t^{\prime } \\ 
	n^{\prime } \\ 
	b^{\prime }%
	\end{array}%
	\right] =\left[ 
	\begin{array}{ccc}
	0 & k & 0 \\ 
	-k & 0 & r \\ 
	0 & -r & 0%
	\end{array}%
	\right] \left[ 
	\begin{array}{c}
	t \\ 
	n \\ 
	b%
	\end{array}%
	\right]
	\end{equation*}%
	where $t=\alpha ^{^{\prime }}$ is unit tangent, $n$ is unit principal
	normal, $b=t\times n$\ is binormal, where $\times $ denotes the quaternion
	product. $k=\left \Vert t^{\prime }\right \Vert $ is the principal curvature
	and $r$ is the torsion of the curve $\gamma .$ Morever these Frenet vectors
	hold the following equations:%
	\begin{eqnarray*}
		h(t,t) &=&h(n,n)=h(b,b)=1, \\
		h(t,n) &=&h(t,b)=h(n,b)=0.
	\end{eqnarray*}
\end{theorem}

\begin{theorem}
	\label{teo 2.2 copy(1)}Let $I=\left[ 0,1\right] $ denote the unit interval
	in the real line $\mathbb{R}$ and 
	\begin{equation*}
	\alpha ^{\left( 4\right) }:I\subset \mathbb{R}\longrightarrow Q,
	\end{equation*}%
	\begin{equation*}
	\text{ \  \  \  \  \  \  \  \  \  \  \  \  \  \  \  \  \  \  \  \  \  \  \  \  \  \  \  \  \  \  \  \  \  \  \
		\  \  \  \  \  \  \  \  \  \  \ }s\longrightarrow \alpha ^{\left( 4\right) }(s)=\alpha
	_{0}^{\left( 4\right) }(s)+\alpha _{1}^{\left( 4\right) }(s)e_{1}+\alpha
	_{2}^{\left( 4\right) }(s)e_{2}+\alpha _{3}^{\left( 4\right) }(s)e_{3}
	\end{equation*}%
	be an arc-length curve in $\mathbb{R}^{4}$. Then Frenet equations of\ $%
	\alpha ^{\left( 4\right) }$ are given by%
	\begin{equation}
	\left[ 
	\begin{array}{c}
	T^{^{\prime }} \\ 
	N_{1}^{^{\prime }} \\ 
	N_{2}^{\prime } \\ 
	N_{3}^{\prime }%
	\end{array}%
	\right] =\left[ 
	\begin{array}{cccc}
	0 & K & 0 & 0 \\ 
	-K & 0 & -r & 0 \\ 
	0 & r & 0 & (K-k) \\ 
	0 & 0 & -(K-k) & 0%
	\end{array}%
	\right] \left[ 
	\begin{array}{c}
	T \\ 
	N_{1} \\ 
	N_{2} \\ 
	N_{3}%
	\end{array}%
	\right]
	\end{equation}%
	where $T=\frac{d\alpha ^{\left( 4\right) }}{ds},$ $N_{1},$ $N_{2},$ $N_{3}$
	are the Frenet vectors of the curve $\alpha ^{\left( 4\right) }$ and $%
	K=\left \Vert T^{^{\prime }}\right \Vert $ is the principal curvature, $-r$
	is the torsion and $(K-k)$ is the bitorsion of the curve $\alpha ^{\left(
		4\right) }$. There exists following relation between the Frenet vectors of $%
	\alpha ^{\left( 4\right) }$ and the Frenet vectors of $\alpha $%
	\begin{equation*}
	N_{1}\left( s\right) =b(s)\times T(s),\text{\ }N_{2}\left( s\right)
	=n(s)\times T(s),\text{ }N_{3}\left( s\right) =t(s)\times T(s)
	\end{equation*}%
	and these Frenet vectors satisfy the following equations:%
	\begin{eqnarray*}
		h(T,T) &=&h(N_{1},N_{1})=h(N_{2},N_{2})=h(N_{3},N_{3})=1, \\
		h(T,N_{1})
		&=&h(T,N_{2})=h(T,N_{3})=h(N_{1},N_{2})=h(N_{1},N_{3})=h(N_{2},N_{3})=0.
	\end{eqnarray*}
\end{theorem}
\section{Characterizations of the Quaternionic $\left( 1,3\right)$-Bertrand curve in Euclidean space $\mathbb{R}^{4}$}

If there exists a quaternionic Bertrand curve in $\mathbb{R}^{4}$, then the
torsion $-r$ or bitorsion $K-k$ vanishes. So we can say that there is no
quaternionic Bertrand curves whose torsion and bitorsion are non-zero. Hence
by using the method which is given by Matsuda and Yorozu \cite{mat},\ we
give the definition of quaternionic $(1,3)-$Bertrand curve according to Type
2-Quaternionic Frame and then obtain a characterization for such curves.

\begin{definition}
	\label{tan 4.1}Let $\alpha ^{\left( 4\right) }:I\subset \mathbb{R}%
	\rightarrow \mathbb{R}^{4}$ \ and $\beta ^{\left( 4\right) }:\bar{I}\subset 
	\mathbb{R}\rightarrow \mathbb{R}^{4}$ be a quaternionic curves. There exists
	a regular $C^{\infty }-$function $\varphi :I\rightarrow \overline{I}%
	,s\rightarrow \varphi (s)=\bar{s}$ such that it corresponds each point $%
	\alpha ^{\left( 4\right) }(s)$ of $\alpha ^{\left( 4\right) }$ to the point $%
	\beta ^{\left( 4\right) }(s)$ of $\beta ^{\left( 4\right) },$ for all $s\in
	I.$ If $\left( 1,3\right) -$normal plane spanned by the normal vectors $%
	N_{1}\left( s\right) $ and $N_{3}\left( s\right) $ at the each point $\alpha
	^{\left( 4\right) }(s)$ of $\alpha ^{\left( 4\right) }$ coincides with $%
	\left( 1,3\right) -$normal plane spanned by the normal vectors $\bar{N}%
	_{1}\left( \bar{s}\right) $ and $\bar{N}_{3}\left( \bar{s}\right) $ at the
	corresponding point $\beta ^{\left( 4\right) }(\bar{s})=\beta ^{\left(
		4\right) }(\varphi (s))$ of $\beta ^{\left( 4\right) }$then we called $%
	\alpha ^{\left( 4\right) }$ is a quaternionic $\left( 1,3\right) -$Bertrand
	curve in $\mathbb{E}^{4}$\ and $\beta ^{\left( 4\right) }$ is called a
	quaternionic $\left( 1,3\right) -$Bertrand mate of $\alpha ^{\left( 4\right)
	}.$
\end{definition}

\begin{theorem}
	\label{teo 4.2}Let $\alpha ^{\left( 4\right) }:I\subset \mathbb{R}%
	\rightarrow \mathbb{R}^{4}$ be a quaternionic curve whose the curvatures
	functions $K,$ $-r,$ $K-k$ and $\alpha :I\subset \mathbb{R}\rightarrow 
	\mathbb{R}^{3}$ be a spatial quaternionic curve associated with quaternionic
	curve $\alpha ^{\left( 4\right) }$ in\ $\mathbb{R}^{4}$ with the curvatures $%
	k$ and $r.$ Then $\alpha ^{\left( 4\right) }$ is a quaternionic $\left(
	1,3\right) -$Bertrand curve if and only if there exists constant real
	numbers $a\neq 0,$ $b\neq 0,$ $c,$ $d$ satifying%
	\begin{equation}
	ar(s)+b\left( K-k\right) (s)\neq 0,
	\end{equation}%
	\begin{equation}
	aK(s)-c\left[ ar(s)+b\left( K-k\right) (s)\right] =1,
	\end{equation}%
	\begin{equation}
	cK(s)+r(s)=d\left( K-k\right) (s)
	\end{equation}%
	\begin{equation}
	\left( 1-c^{2}\right) K(s)r(s)+c\left( K^{2}\left( s\right) -r^{2}(s)-\left(
	K-k\right) ^{2}(s)\right) \neq 0,
	\end{equation}
\end{theorem}

for all $s\in I.$

\begin{proof}
	We suppose that $\alpha ^{\left( 4\right) }$ is a quaternionic $\left(
	1,3\right) $ Bertrand curve given by arc-lenght parameter $s$ and $\beta
	^{\left( 4\right) }$ is a quaternionic $\left( 1,3\right) $-Bertrand mate of 
	$\alpha ^{\left( 4\right) }$ with arc-lenght parameter $\bar{s}.$ Then we
	have 
	\begin{equation}
	\beta ^{\left( 4\right) }\left( \bar{s}\right) =\beta ^{\left( 4\right)
	}\left( \varphi \left( s\right) \right) =\alpha ^{\left( 4\right)
	}(s)+a(s)N_{1}(s)+b(s)N_{3}(s)
	\end{equation}%
	for all $s\in I,$ where $a,$ $b:I\rightarrow \mathbb{R}$ are differentiable
	functions. Taking the derivative of (7) with respect to $s$ and using (2),
	we have%
	\begin{equation}
	\begin{array}{cl}
	\bar{T}\left( \bar{s}\right) \varphi ^{\shortmid }\left( s\right) = & \left[
	1-a(s)K(s)\right] T\left( s\right) +a^{\shortmid }(s)N_{1}(s) \\ 
	& -\left[ a(s)r(s)+b(s)(K-k)(s)\right] N_{2}(s)+b^{\shortmid }(s)N_{3}(s)%
	\end{array}%
	\end{equation}%
	for all $s\in I.$
	
	Since $Sp\left \{ N_{1}(s),N_{3}(s)\right \} =Sp\left \{ \bar{N}_{1}(\bar{s}%
	),\bar{N}_{3}(\bar{s})\right \} ,$ we can write%
	\begin{equation}
	\bar{N}_{1}(\bar{s})=\cos \theta (s)N_{1}(s)+\sin \theta (s)N_{3}(s),
	\end{equation}%
	\begin{equation}
	\bar{N}_{3}(\bar{s})=-\sin \theta (s)N_{1}(s)+\cos \theta (s)N_{3}(s).
	\end{equation}%
	We notice that $\sin \theta (s)\neq 0.$\ Otherwise, $\bar{N}_{1}(\bar{s}%
	)=\pm N_{1}(s).$ By using (8) and (9), we get%
	\begin{equation}
	h(\bar{T}\left( \bar{s}\right) \varphi ^{\shortmid }\left( s\right) ,\bar{N}%
	_{1}(\bar{s}))=\cos \theta (s)a^{\shortmid }(s)+\sin \theta (s)b^{\shortmid
	}(s)=0.
	\end{equation}%
	By using (8) and (10), we get%
	\begin{equation}
	h(\bar{T}\left( \bar{s}\right) \varphi ^{\shortmid }\left( s\right) ,\bar{N}%
	_{3}(\bar{s}))=-\sin \theta (s)a^{\shortmid }(s)+\cos \theta (s)b^{\shortmid
	}(s)=0.
	\end{equation}%
	From (11) and (12), since $\left \vert 
	\begin{array}{cc}
	\cos \theta (s) & \sin \theta (s) \\ 
	-\sin \theta (s) & \cos \theta (s)%
	\end{array}%
	\right \vert =1,$ we find 
	\begin{equation*}
	a^{\shortmid }(s)=0\text{, }b^{\shortmid }(s)=0.
	\end{equation*}%
	From above equalites, we obtain that $a$ and $b$ are real constants.
	
	So, we can rewrite $\beta ^{\left( 4\right) }$ given by (7) as:%
	\begin{equation}
	\beta ^{\left( 4\right) }\left( \bar{s}\right) =\alpha ^{\left( 4\right)
	}(s)+aN_{1}(s)+bN_{3}(s)
	\end{equation}%
	and the unit tangent vector of $\beta ^{\left( 4\right) }$ is following:%
	\begin{equation}
	\bar{T}\left( \bar{s}\right) \varphi ^{\shortmid }\left( s\right) =\left(
	1-aK(s)\right) T\left( s\right) -\left( ar(s)+b(K-k)(s)\right) N_{2}(s),
	\end{equation}%
	where%
	\begin{equation}
	\left( \varphi ^{\shortmid }\left( s\right) \right) ^{2}=\left(
	1-aK(s)\right) ^{2}+\left( ar(s)+b(K-k)(s)\right) ^{2}\neq 0
	\end{equation}%
	for all $s\in I,$ if we denote by 
	\begin{equation}
	\cos \tau \left( s\right) =\left( \frac{1-aK(s)}{\varphi ^{\shortmid }\left(
		s\right) }\right) \text{, }\sin \tau \left( s\right) =-\left( \frac{%
		ar(s)+b(K-k)(s)}{\varphi ^{\shortmid }\left( s\right) }\right) ,
	\end{equation}%
	where $\tau $ is differentiable function on $I$, so we can rewrite (14) as:%
	\begin{equation}
	\bar{T}\left( \bar{s}\right) =\cos \tau \left( s\right) T(s)+\sin \tau
	\left( s\right) N_{2}(s)
	\end{equation}%
	If we calculate the derivative of (17) with respect to $s$ and use (2), we
	obtain%
	\begin{equation*}
	\begin{array}{c}
	\bar{K}(\bar{s})\bar{N}_{1}\left( \bar{s}\right) \varphi ^{\shortmid }\left(
	s\right) =\left( \cos \tau \left( s\right) \right) ^{\shortmid }T(s)+\left[
	\cos \tau \left( s\right) K(s)+\sin \tau \left( s\right) r(s)\right] N_{1}(s)
	\\ 
	+\left( \sin \tau \left( s\right) \right) ^{\shortmid }N_{2}(s)+\sin \tau
	\left( s\right) (K-k)(s)N_{3}(s)%
	\end{array}%
	\end{equation*}%
	From (9), we know that $\bar{N}_{1}\left( \bar{s}\right) \in Sp\left \{
	N_{1}(s),N_{3}(s)\right \} .$ So, from the above equation%
	\begin{equation*}
	\left( \cos \tau \left( s\right) \right) ^{\shortmid }=0\text{, }\left( \sin
	\tau \left( s\right) \right) ^{\shortmid }=0,
	\end{equation*}%
	and it means that $\tau =\tau _{0}$\ is a real constant. Then we can rewrite
	(17) as:%
	\begin{equation}
	\bar{T}\left( \bar{s}\right) =\cos \tau _{0}\left( s\right) T(s)+\sin \tau
	_{0}\left( s\right) N_{2}(s)
	\end{equation}%
	and from (16), we get%
	\begin{equation}
	\cos \tau _{0}\varphi ^{\shortmid }\left( s\right) =1-aK(s)
	\end{equation}%
	and%
	\begin{equation}
	\sin \tau _{0}\varphi ^{\shortmid }\left( s\right) =-\left(
	ar(s)+b(K-k)(s)\right)
	\end{equation}%
	From (19) and (20)%
	\begin{equation}
	\left( 1-aK(s)\right) \sin \tau _{0}=-\left( ar(s)+b(K-k)(s)\right) \cos
	\tau _{0}
	\end{equation}%
	If $\sin \tau _{0}$ vanishes, then $\cos \tau _{0}=\pm 1.$ And from (18), we
	get $\bar{T}\left( \bar{s}\right) =\pm T(s)$. If we differentiate this
	equality and use (2),\ we have $\bar{N}_{1}(\bar{s})=\pm 1N_{1}(s).$ It is a
	contradiction. So $\sin \tau _{0}\neq 0,$ that is, from (20) implies that 
	\begin{equation*}
	ar(s)+b(K-k)(s)\neq 0.
	\end{equation*}
	Hence we obtain the relation (3).
	
	If we denote the constant $c$ by $c=\frac{\cos \tau _{0}}{\sin \tau _{0}}$,
	from (21), 
	\begin{equation*}
	aK(s)-c\left( ar(s)+b(K-k)(s)\right) =1
	\end{equation*}%
	for all $s\in I.$ Thus we find the relation (4). Differentiating (18) with
	respect to $s$ and using the equations of Type 2- Quaternionic Frame given
	by (2), we have 
	\begin{equation}
	\bar{K}(\bar{s})\bar{N}_{1}\left( \bar{s}\right) \varphi ^{\shortmid }\left(
	s\right) =\left( \cos \tau _{0}K(s)+\sin \tau _{0}r(s)\right) N_{1}(s)+\sin
	\tau _{0}(K-k)(s)N_{3}(s).
	\end{equation}%
	By using (22) we have%
	\begin{equation*}
	\left( \bar{K}(\bar{s})\varphi ^{\shortmid }\left( s\right) \right)
	^{2}=\left( \sin \tau _{0}\right) ^{2}\left[ \left( \frac{\cos \tau _{0}}{%
		\sin \tau _{0}}K(s)+r(s)\right) ^{2}+\left( (K-k)(s)\right) ^{2}\right] .
	\end{equation*}%
	By using (19) and (20) in above equality,%
	\begin{equation}
	\left( \bar{K}(\bar{s})\varphi ^{\shortmid }\left( s\right) \right)
	^{2}=\left( ar(s)+b(K-k)(s)\right) ^{2}\left[ \left( cK(s)+r(s)\right)
	^{2}+\left( (K-k)(s)\right) ^{2}\right] \left( \varphi ^{\shortmid }\left(
	s\right) \right) ^{-2}.
	\end{equation}%
	On the other hand, from (4) and (15), we obtain%
	\begin{equation}
	\left( \varphi ^{\shortmid }\left( s\right) \right) ^{2}=\left(
	1+c^{2}\right) \left( ar(s)+b(K-k)(s)\right) ^{2}
	\end{equation}%
	Then if we consider with (23) and (24), we get%
	\begin{equation}
	\left( \bar{K}(\bar{s})\varphi ^{\shortmid }\left( s\right) \right) ^{2}=%
	\frac{1}{1+c^{2}}\left[ \left( cK(s)+r(s)\right) ^{2}+\left( (K-k)(s)\right)
	^{2}\right]
	\end{equation}%
	By using (19), (20) and the ralation (4), we rewrite (22) as:%
	\begin{equation}
	\bar{N}_{1}\left( \bar{s}\right) =\cos \eta (s)N_{1}(s)+\sin \eta
	(s)N_{3}(s),
	\end{equation}%
	where 
	\begin{equation}
	\cos \eta (s)=\frac{-\left( ar(s)+b(K-k)(s)\right) \left( cK(s)+r(s)\right) 
	}{\bar{K}(\bar{s})\left( \varphi ^{\shortmid }\left( s\right) \right) ^{2}}%
	\text{,}
	\end{equation}%
	and 
	\begin{equation}
	\sin \eta (s)=\frac{-\left( ar(s)+b(K-k)(s)\right) (K-k)(s)}{\bar{K}(\bar{s}%
		)\left( \varphi ^{\shortmid }\left( s\right) \right) ^{2}}
	\end{equation}%
	for $s\in I.$ Here, $\eta $ is differentiable function on $I.$
	
	Taking the derivative of (26) and using the equations of Type 2-
	Quaternionic Frame given by (2), we have
	
	\begin{eqnarray}
	\left( -\bar{K}(\bar{s})\bar{T}\left( \bar{s}\right) -\bar{r}\left( \bar{s}%
	\right) \bar{N}_{2}\left( \bar{s}\right) \right) \varphi ^{\shortmid }\left(
	s\right) &=&-\cos \eta (s)K(s)T(s)+\left( \cos \eta (s)\right) ^{^{\shortmid
	}}N_{1}(s) \\
	&&+\left( -\cos \eta (s)r(s)-\sin \eta (s)(K-k)(s)\right) N_{2}(s)  \notag \\
	&&+\left( \sin \eta (s)\right) ^{\shortmid }N_{3}(s)  \notag
	\end{eqnarray}%
	From (29), it satisfies%
	\begin{equation*}
	\left( \cos \eta (s)\right) ^{\shortmid }=0\text{ \ and }\left( \sin \eta
	(s)\right) ^{\shortmid }=0,
	\end{equation*}
	
	that is, $\eta =\eta _{0}$ is a constant function on $I.$ Let $d=\frac{\cos
		\eta _{0}}{\sin \eta _{0}}$ be a constant then from (27) and (28), we find
	following relation:%
	\begin{equation*}
	cK(s)+r(s)=d(K-k)(s).
	\end{equation*}%
	Thus we obtain the relation (5).
	
	Since $\eta =\eta _{0}$ is a constant function, we rewrite (29)%
	\begin{eqnarray*}
		\left( -\bar{K}(\bar{s})\bar{T}\left( \bar{s}\right) -\bar{r}\left( \bar{s}%
		\right) \bar{N}_{2}\left( \bar{s}\right) \right) \varphi ^{\shortmid }\left(
		s\right) &=&-\cos \eta _{0}K(s)T(s)+ \\
		&&+\left( -\cos \eta _{0}r(s)-\sin \eta _{0}(K-k)(s)\right) N_{2}(s)
	\end{eqnarray*}
	
	By considering above equation with (14), we get 
	\begin{eqnarray*}
		-\bar{r}\left( \bar{s}\right) \bar{N}_{2}\left( \bar{s}\right) \varphi
		^{\shortmid }\left( s\right) &=&\left( \bar{K}(\bar{s})\varphi ^{\shortmid
		}\left( s\right) \frac{\left( 1-aK(s)\right) }{\varphi ^{\shortmid }\left(
			s\right) }-\cos \eta _{0}K(s)\right) T(s) \\
		&&+\left( 
		\begin{array}{c}
			-\bar{K}(\bar{s})\varphi ^{\shortmid }\left( s\right) \frac{\left(
				ar(s)+b(K-k)(s)\right) }{\varphi ^{\shortmid }\left( s\right) } \\ 
			-\cos \eta _{0}r(s)-\sin \eta _{0}(K-k)(s)%
		\end{array}%
		\right) N_{2}(s) \\
		&=&\frac{1}{\bar{K}(\bar{s})\left( \varphi ^{\shortmid }\left( s\right)
			\right) ^{2}}\left \{ A(s)T(s)+B(s)N_{2}(s)\right \} ,
	\end{eqnarray*}%
	where%
	\begin{eqnarray*}
		A(s) &=&\left( \bar{K}(\bar{s})\varphi ^{\shortmid }\left( s\right) \right)
		^{2}\left( 1-aK(s)\right) +\left( ar(s)+b(K-k)(s)\right) \left(
		cK(s)+r(s)\right) K(s), \\
		B(s) &=&-\left( \bar{K}(\bar{s})\varphi ^{\shortmid }\left( s\right) \right)
		^{2}\left( ar(s)+b(K-k)(s)\right) +\left( ar(s)+b(K-k)(s)\right) \left(
		cK(s)+r(s)\right) r(s) \\
		&&+\left( ar(s)+b(K-k)(s)\right) \left( \left( K-k\right) (s)\right) ^{2}
	\end{eqnarray*}%
	By using (25) and the ralation (4), we can rewrite $A(s)$ and $B(s)$ as
	follow:%
	\begin{equation*}
	A(s)=\left( 1+c^{2}\right) ^{-1}\left( ar(s)+b(K-k)(s)\right) \left \{
	\left( 1-c^{2}\right) K\left( s\right) r(s)+c\left( K^{2}\left( s\right)
	-r^{2}(s)-\left( K-k\right) ^{2}(s)\right) \right \}
	\end{equation*}%
	and 
	\begin{equation*}
	B(s)=-c\left( 1+c^{2}\right) ^{-1}\left( ar(s)+b(K-k)(s)\right) \left \{
	\left( 1-c^{2}\right) K\left( s\right) r(s)+c\left( K^{2}\left( s\right)
	-r^{2}(s)-\left( K-k\right) ^{2}(s)\right) \right \}
	\end{equation*}%
	Since $\bar{r}\left( \bar{s}\right) \bar{N}_{2}\left( \bar{s}\right) \varphi
	^{\shortmid }\left( s\right) \neq 0$ for $\forall s\in I$ , we have%
	\begin{equation*}
	\left( 1-c^{2}\right) K\left( s\right) r(s)+c\left( K^{2}\left( s\right)
	-r^{2}(s)-\left( K-k\right) ^{2}(s)\right) \neq 0
	\end{equation*}%
	for all $s\in I.$ Thus we obtain the relation (6).
	
	Conversely, let $\alpha ^{\left( 4\right) }:I\subset \mathbb{R}\rightarrow $ 
	$\mathbb{E}^{4}$ be a quaternionic curve with curvatures $K,$ $-r,$ $%
	(K-k)\neq 0$ satisfying the equations (3), (4), (5), (6) for constant
	numbers $a,$ $b,c,$ $d$ and $\beta ^{\left( 4\right) }$ be a quaternionic
	curve such that 
	\begin{equation*}
	\beta ^{\left( 4\right) }(s)=\alpha ^{\left( 4\right) }\left( s\right)
	+aN_{1}(s)+bN_{3}(s)
	\end{equation*}%
	for all $s\in I.$ Differentiating above equality with respect to $s$ and
	using the equations of Type 2- Quaternionic Frame given by (2), we have%
	\begin{equation*}
	\frac{d\beta ^{\left( 4\right) }\left( s\right) }{ds}=\left( 1-aK(s)\right)
	T\left( s\right) -\left( ar(s)+b(K-k)(s)\right) N_{2}(s),
	\end{equation*}%
	thus, by using the relation \ (4), we obtain%
	\begin{equation*}
	\frac{d\beta ^{\left( 4\right) }\left( s\right) }{ds}=-\left(
	ar(s)+b(K-k)(s)\right) (cT\left( s\right) +N_{2}(s))
	\end{equation*}%
	for all $s\in I$. From the relation (3), since $ar(s)+b(K-k)(s)\neq 0,$ the
	curve $\beta ^{\left( 4\right) }$ is a regular curve. Then there exists a
	regular $C^{\infty }-$function $\varphi :I\rightarrow \bar{I}$ defined by%
	\begin{equation*}
	\bar{s}=\varphi \left( s\right) =\int \left \Vert \frac{d\beta ^{\left(
			4\right) }\left( s\right) }{ds}\right \Vert ds
	\end{equation*}%
	where $\bar{s}$ denotes the arc-length parameter of $\beta ^{\left( 4\right)
	}$. Then%
	\begin{equation}
	\varphi ^{\shortmid }\left( s\right) =\varepsilon \sqrt{1+c^{2}}\left(
	ar(s)+b(K-k)(s)\right)
	\end{equation}%
	where if $ar(s)+b(K-k)(s)>0$ then $\varepsilon =1,$\ if $ar(s)+b(K-k)(s)<0$\
	then $\varepsilon =-1$ for all $s\in I$. Hence we can express $\beta
	^{\left( 4\right) }$ again as:%
	\begin{equation*}
	\beta ^{\left( 4\right) }\left( \bar{s}\right) =\beta ^{\left( 4\right)
	}\left( \varphi \left( s\right) \right) =\alpha ^{\left( 4\right)
	}(s)+aN_{1}(s)+bN_{3}(s)
	\end{equation*}%
	Differentiating the above equality with respect to $s,$ we have%
	\begin{equation}
	\varphi ^{\shortmid }\left( s\right) \frac{d\beta ^{\left( 4\right) }\left( 
		\bar{s}\right) }{d\bar{s}}=-\left( ar(s)+b(K-k)(s)\right) \left(
	cT(s)+N_{2}(s)\right)
	\end{equation}%
	Considering (30) and (31) with together, we can write%
	\begin{equation}
	\bar{T}\left( \bar{s}\right) =\frac{1}{\epsilon \sqrt{1+c^{2}}}\left(
	cT(s)+N_{2}(s)\right) ,
	\end{equation}%
	where $\epsilon =-\varepsilon .$\ Differentiating (32) with respect to $s$
	and using the equations of Type 2-Quaternionic Frame, we get%
	\begin{equation*}
	\varphi ^{\shortmid }\left( s\right) \frac{d\bar{T}\left( \bar{s}\right) }{d%
		\bar{s}}=\frac{1}{\epsilon \sqrt{1+c^{2}}}\left( \left( cK(s)+r(s)\right)
	N_{1}\left( s\right) +(K-k)(s)N_{3}\left( s\right) \right)
	\end{equation*}%
	Then we can calculate curvature of $\beta ^{\left( 4\right) }$ as: 
	\begin{equation}
	\bar{K}(\bar{s})=\left \Vert \frac{d\bar{T}\left( \bar{s}\right) }{d\bar{s}}%
	\right \Vert =\frac{\sqrt{\left( cK(s)+r(s)\right) ^{2}+\left( \left(
			K-k\right) \left( s\right) \right) ^{2}}}{\varphi ^{\shortmid }(s)\sqrt{%
			1+c^{2}}}.
	\end{equation}%
	for all $s\in I$. From using the equations of Type 2-Quaternionic Frame
	given by (2), we can determine the unit normal vector $\bar{N}_{1}$ along $%
	\beta ^{\left( 4\right) }$%
	\begin{eqnarray*}
		\bar{N}_{1}(\bar{s}) &=&\frac{1}{\bar{K}(\bar{s})}\frac{d\bar{T}\left( \bar{s%
			}\right) }{d\bar{s}} \\
		&=&\frac{\left( \left( cK(s)+r(s)\right) N_{1}\left( s\right)
			+(K-k)(s)N_{3}\left( s\right) \right) }{\epsilon \sqrt{\left(
				cK(s)+r(s)\right) ^{2}+\left( \left( K-k\right) \left( s\right) \right) ^{2}}%
		}
	\end{eqnarray*}%
	for all $s\in I$. Thus we can put%
	\begin{equation}
	\bar{N}_{1}(\bar{s})=\cos \gamma (s)N_{1}\left( s\right) +\sin \gamma
	(s)N_{3}\left( s\right) ,
	\end{equation}%
	where%
	\begin{equation}
	\cos \gamma (s)=\frac{cK(s)+r(s)}{\epsilon \sqrt{\left( cK(s)+r(s)\right)
			^{2}+\left( \left( K-k\right) \left( s\right) \right) ^{2}}}
	\end{equation}%
	and%
	\begin{equation}
	\sin \gamma (s)=\frac{\left( K-k\right) \left( s\right) }{\epsilon \sqrt{%
			\left( cK(s)+r(s)\right) ^{2}+\left( \left( K-k\right) \left( s\right)
			\right) ^{2}}},
	\end{equation}%
	So differentiating (34) with respect to $s$ and using (2),\ we have 
	\begin{eqnarray*}
		\frac{\bar{N}_{1}(\bar{s})}{d\bar{s}}\varphi ^{\shortmid }\left( s\right)
		&=&-K(s)\cos \gamma (s)T(s)+\left( \cos \gamma (s)\right) ^{\shortmid
		}N_{1}\left( s\right) \\
		&&+\left( -r\left( s\right) \cos \gamma (s)-\left( K-k\right) \left(
		s\right) \sin \gamma (s)\right) N_{2}\left( s\right) +\left( \sin \gamma
		(s)\right) ^{\shortmid }N_{3}\left( s\right)
	\end{eqnarray*}%
	On the other hand, from the relation (5), we get 
	\begin{equation*}
	\frac{cK(s)+r(s)}{\left( K-k\right) \left( s\right) }=d
	\end{equation*}%
	Calculating the derivative of the last equation with respect to $s,$ we find
	the following equality:%
	\begin{equation}
	\left( cK^{^{\shortmid }}(s)+r^{^{\shortmid }}(s)\right) \left( K-k\right)
	\left( s\right) -\left( cK(s)+r(s)\right) \left( K-k\right) ^{\shortmid
	}\left( s\right) =0
	\end{equation}%
	Taking the derivatives of (35) and (36) and using (37), we obtain 
	\begin{equation*}
	\left( \cos \gamma (s)\right) ^{\shortmid }=0\text{ and }\left( \sin \gamma
	(s)\right) ^{\shortmid }=0,
	\end{equation*}%
	that is, $\gamma $ is a real constant with value $\gamma _{0}.$\ Thus we
	have 
	\begin{equation}
	\cos \gamma _{0}=\frac{cK(s)+r(s)}{\epsilon \sqrt{\left( cK(s)+r(s)\right)
			^{2}+\left( K-k\right) ^{2}\left( s\right) }}
	\end{equation}%
	and%
	\begin{equation}
	\sin \gamma _{0}=\frac{\left( K-k\right) \left( s\right) }{\epsilon \sqrt{%
			\left( cK(s)+r(s)\right) ^{2}+\left( K-k\right) ^{2}\left( s\right) }}
	\end{equation}%
	Hence we can rewrite (34) as:%
	\begin{equation}
	\bar{N}_{1}(\bar{s})=\cos \gamma _{0}N_{1}\left( s\right) +\sin \gamma
	_{0}N_{3}\left( s\right)
	\end{equation}%
	Differentiating (40) with respect to $s$ and using the equations of Type 2-
	Quaternionic Frame given by (2), (38), (39), we have%
	\begin{eqnarray*}
		\frac{d\bar{N}_{1}(\bar{s})}{d\bar{s}} &=&-\frac{\left( cK(s)+r(s)\right)
			K(s)}{\epsilon \varphi ^{\shortmid }\left( s\right) \sqrt{\left(
				cK(s)+r(s)\right) ^{2}+\left( \left( K-k\right) \left( s\right) \right) ^{2}}%
		}T(s) \\
		&&-\frac{\left( cK(s)+r(s)\right) r(s)+\left( \left( K-k\right) \left(
			s\right) \right) ^{2}}{\epsilon \varphi ^{\shortmid }\left( s\right) \sqrt{%
				\left( cK(s)+r(s)\right) ^{2}+\left( \left( K-k\right) \left( s\right)
				\right) ^{2}}}N_{2}(s)
	\end{eqnarray*}%
	By using (32) and (33), we have%
	\begin{equation*}
	\bar{K}(\bar{s})\bar{T}(\bar{s})=\frac{\left( cK(s)+r(s)\right) ^{2}+\left(
		\left( K-k\right) \left( s\right) \right) ^{2}}{\epsilon \varphi ^{\shortmid
		}\left( s\right) \left( 1+c^{2}\right) \sqrt{\left( cK(s)+r(s)\right)
			^{2}+\left( \left( K-k\right) \left( s\right) \right) ^{2}}}\left(
	cT(s)+N_{2}(s)\right)
	\end{equation*}%
	By using the above equalities, we have 
	\begin{equation*}
	\frac{d\bar{N}_{1}(\bar{s})}{d\bar{s}}+\bar{K}(\bar{s})\bar{T}(\bar{s})=%
	\frac{P(s)}{R(s)}T(s)+\frac{Q(s)}{R(s)}N_{2}(s),
	\end{equation*}%
	where we can easily show%
	\begin{eqnarray*}
		P(s) &=&-\left[ \left( 1-c^{2}\right) K(s)r(s)+c\left \{
		K^{2}(s)-r^{2}(s)-\left( K-k\right) ^{2}\left( s\right) \right \} \right] \\
		Q(s) &=&c\left[ \left( 1-c^{2}\right) K(s)r(s)+c\left \{
		K^{2}(s)-r^{2}(s)-\left( K-k\right) ^{2}\left( s\right) \right \} \right] \\
		R(s) &=&\epsilon \varphi ^{\shortmid }\left( s\right) \left( 1+c^{2}\right) 
		\sqrt{\left( cK(s)+r(s)\right) ^{2}+\left( \left( K-k\right) \left( s\right)
			\right) ^{2}}\neq 0\text{.}
	\end{eqnarray*}%
	Since $\frac{d\bar{N}_{1}(\bar{s})}{d\bar{s}}+\bar{K}(\bar{s})\bar{T}(\bar{s}%
	)=-\bar{r}\left( \bar{s}\right) \bar{N}_{2}(\bar{s}),$ we obtain the torsion
	of $\beta ^{\left( 4\right) }$%
	\begin{eqnarray}
	-\bar{r}\left( \bar{s}\right) &=&\left \Vert \frac{d\bar{N}_{1}(\bar{s})}{d%
		\bar{s}}+\bar{K}(\bar{s})\bar{T}(\bar{s})\right \Vert \\
	&=&\frac{1}{R(s)}\sqrt{P^{2}(s)+Q^{2}(s)}  \notag \\
	&=&\frac{\left \vert \left( 1-c^{2}\right) K(s)r(s)+c\left \{
		K^{2}(s)-r^{2}(s)-\left( K-k\right) ^{2}\left( s\right) \right \} \right
		\vert }{\varphi ^{\shortmid }\left( s\right) \sqrt{1+c^{2}}\sqrt{\left(
			cK(s)+r(s)\right) ^{2}+\left( \left( K-k\right) \left( s\right) \right) ^{2}}%
	}.  \notag
	\end{eqnarray}%
	Now we can define unit vector field $\bar{N}_{2}(\bar{s})$ along $\beta
	^{\left( 4\right) },$%
	\begin{equation*}
	\bar{N}_{2}(\bar{s})=-\frac{1}{\bar{r}\left( \bar{s}\right) }\left( \frac{d%
		\bar{N}_{1}(\bar{s})}{d\bar{s}}+\bar{K}(\bar{s})\bar{T}(\bar{s})\right) ,
	\end{equation*}%
	that is,%
	\begin{equation}
	\bar{N}_{2}(\bar{s})=\frac{1}{\epsilon \sqrt{1+c^{2}}}\left(
	-T(s)+cN_{2}(s)\right)
	\end{equation}%
	Also, we can define the unit vector field $\bar{N}_{3}(\bar{s})$ along $%
	\beta ^{\left( 4\right) }$ as: 
	\begin{eqnarray}
	\bar{N}_{3}(\bar{s}) &=&-\sin \gamma _{0}N_{1}\left( s\right) +\cos \gamma
	_{0}N_{3}\left( s\right)  \notag \\
	&=&\frac{1}{\epsilon \sqrt{\left( cK(s)+r(s)\right) ^{2}+\left( \left(
			K-k\right) \left( s\right) \right) ^{2}}}\left( 
	\begin{array}{c}
	-\left( K-k\right) \left( s\right) N_{1}\left( s\right) \\ 
	+\left( cK(s)+r(s)\right) N_{3}\left( s\right)%
	\end{array}%
	\right)
	\end{eqnarray}%
	Finally we define the bitorsion of $\beta ^{\left( 4\right) }$%
	\begin{eqnarray}
	\left( \bar{K}-\bar{k}\right) (\bar{s}) &=&\left \langle \frac{d\bar{N}_{2}(%
		\bar{s})}{d\bar{s}},\bar{N}_{3}(\bar{s})\right \rangle  \notag \\
	&=&\frac{\left( K-k\right) \left( s\right) K(s)\sqrt{1+c^{2}}}{\varphi
		^{\shortmid }\left( s\right) \sqrt{\left( cK(s)+r(s)\right) ^{2}+\left(
			\left( K-k\right) \left( s\right) \right) ^{2}}}
	\end{eqnarray}%
	for all $s\in I.$ Using the Frenet vectors $\bar{T},$ $\bar{N}_{1},$ $\bar{N}%
	_{2},$ $\bar{N}_{3}$ we can easily see that 
	\begin{equation*}
	h\left( \bar{T},\bar{T}\right) =h\left( \bar{N}_{1},\bar{N}_{1}\right)
	=h\left( \bar{N}_{2},\bar{N}_{2}\right) =h\left( \bar{N}_{3},\bar{N}%
	_{3}\right) =1,
	\end{equation*}%
	and 
	\begin{equation*}
	h\left( \bar{T},\bar{N}_{1}\right) =h\left( \bar{T},\bar{N}_{2}\right)
	=h\left( \bar{T},\bar{N}_{3}\right) =h\left( \bar{N}_{1},\bar{N}_{2}\right)
	=h\left( \bar{N}_{1},\bar{N}_{3}\right) =h\left( \bar{N}_{2},\bar{N}%
	_{3}\right) =0,
	\end{equation*}%
	for all $s\in I$ where $\left \{ \bar{T}(\bar{s}),\bar{N}_{1}(\bar{s}),\bar{N%
	}_{2}(\bar{s}),\bar{N}_{3}(\bar{s})\right \} $ is Frenet frame along
	quaternionic curve $\beta ^{4}$ in $\mathbb{E}^{4}.$ And it is fact that $%
	\left( 1,3\right) $ normal plane $Sp\left \{ N_{1},N_{3}\right \} $ of $%
	\alpha ^{\left( 4\right) }$ coincides $\left( 1,3\right) $ normal plane $%
	Sp\left \{ \bar{N}_{1},\bar{N}_{3}\right \} $ of $\beta ^{\left( 4\right) }$%
	. Consequently, $\alpha ^{\left( 4\right) }$ is a quaternionic $\left(
	1,3\right) $ Bertrand curve in $\mathbb{E}^{4}$ and $\beta ^{\left( 4\right)
	}$ is quaternionic $\left( 1,3\right) $ Bertrand mate of it. This completes
	the proof.
\end{proof}

\begin{theorem}
	\label{teo 4.3}Let $\alpha ^{\left( 4\right) }:I\subset \mathbb{R}%
	\rightarrow $ $\mathbb{E}^{4}$ be a quaternionic $\left( 1,3\right) $
	Bertrand curve\ and $\beta ^{\left( 4\right) }$ be a quaternionic $\left(
	1,3\right) $ Bertrand mate of $\alpha ^{\left( 4\right) }$ and $\varphi
	:I\rightarrow \bar{I},$ $\bar{s}=\varphi (s)$ is a regular $C^{\infty }-$%
	function such that $s$ and $\bar{s}$ are arc-length parameter of $\alpha
	^{\left( 4\right) }$ and $\beta ^{\left( 4\right) }$, respectively. Then the
	distance between the points $\alpha ^{\left( 4\right) }(s)$ and $\beta
	^{\left( 4\right) }(\bar{s})$ is constant for all $s\in I.$
\end{theorem}

\begin{proof}
	Let $\alpha ^{\left( 4\right) }:I\subset \mathbb{R}\rightarrow $ $\mathbb{E}%
	^{4}$ be quaternionic $\left( 1,3\right) $-Bertrand curve in $\mathbb{E}^{4}$
	and $\beta ^{\left( 4\right) }:\bar{I}\subset \mathbb{R}\rightarrow \mathbb{E%
	}^{4}$ be a quaternionic $\left( 1,3\right) $-Bertrand mate of $\alpha
	^{\left( 4\right) }.$ Then we can write,%
	\begin{equation*}
	\beta ^{\left( 4\right) }\left( \bar{s}\right) =\alpha ^{\left( 4\right)
	}(s)+aN_{1}(s)+bN_{3}(s)
	\end{equation*}
	
	where $a$ and $b$ are non-zero constants. Thus, we can write%
	\begin{equation*}
	\beta ^{\left( 4\right) }\left( \bar{s}\right) -\alpha ^{\left( 4\right)
	}(s)=aN_{1}(s)+bN_{3}(s)
	\end{equation*}%
	and%
	\begin{equation*}
	\left \Vert \beta ^{\left( 4\right) }\left( \bar{s}\right) -\alpha ^{\left(
		4\right) }(s)\right \Vert =\sqrt{a^{2}+b^{2}}.
	\end{equation*}
\end{proof}

\begin{theorem}
	\label{cor 4.1}Let $\alpha ^{\left( 4\right) }:I\subset \mathbb{R}%
	\rightarrow $ $\mathbb{E}^{4}$ be a quaternionic $\left( 1,3\right) $%
	-Bertrand curve such that $\alpha :I\subset \mathbb{R}\rightarrow $ $\mathbb{%
		E}^{3}$ is a spatial quaternionic curve associated with $\alpha ^{\left(
		4\right) }$. If $\beta ^{\left( 4\right) }$ is a quaternionic $\left(
	1,3\right) $-Bertrand mate of $\alpha ^{\left( 4\right) }$ then the
	curvature functions of $\beta ^{\left( 4\right) }$\ are determined in terms
	of the principal curvature $K\ $of the curve $\alpha ^{\left( 4\right) }$
	and the principal curvature $k\ $of the curve $\alpha $ as follows:%
	\begin{eqnarray*}
		\bar{K}(\bar{s}) &=&\frac{c\sqrt{1+d^{2}}\left( K-k\right) \left( s\right) }{%
			\epsilon \delta \left( 1+c^{2}\right) \left( 1-aK(s)\right) }, \\
		-\bar{r}(\bar{s}) &=&\frac{c\left \vert \left( c\left( 1+d^{2}\right) \left(
			K-k\right) \left( s\right) -\left( 1+c^{2}\right) dK(s)\right) \right \vert 
		}{\epsilon \left( 1+c^{2}\right) \sqrt{1+d^{2}}\left( 1-aK(s)\right) }, \\
		\bar{K}(\bar{s})-\bar{k}(\bar{s}) &=&\frac{cK\left( s\right) }{\epsilon
			\delta \sqrt{1+d^{2}}\left( 1-aK(s)\right) },
	\end{eqnarray*}%
	where $\delta $ is the signature of the curvature $K-k,$ that is, $\delta
	\left( K-k\right) >0.$
\end{theorem}

\begin{proof}
	We suppose that $\alpha ^{\left( 4\right) }:I\subset \mathbb{R}\rightarrow $ 
	$\mathbb{E}^{4}$ is a quaternionic curve whose the curvatures functions $K,$ 
	$-r,$ $K-k$ and $\alpha :I\subset \mathbb{R}\rightarrow $ $\mathbb{E}^{3}$
	be a spatial quaternionic curve associated with quaternionic curve $\alpha
	^{\left( 4\right) }$ in\ $\mathbb{E}^{4}$ with the curvatures $k$ and $r.$
	In that case for constant real numbers $a\neq 0,$ $b\neq 0,$ $c,$ $d$ hold
	the relations (3), (4), (5) and (6). If $\beta ^{\left( 4\right) }$ is a
	quaternionic $\left( 1,3\right) $-Bertrand mate of $\alpha ^{\left( 4\right)
	}$ then the curvature functions of $\beta ^{\left( 4\right) }$\ are defined
	by the equations (33), (41) and (44) in Theorem 3.1. If we consider (33),
	(41) and (44) with the relations (3), (4), (5) and (6), we obtain these
	curvature functions in terms of the principal curvature $K\ $of the curve $%
	\alpha ^{\left( 4\right) }$ and the principal curvature $k\ $of the curve $%
	\alpha $.
\end{proof}

\begin{remark}
	We note that if $\alpha ^{\left( 4\right) }$ is a quaternionic $\left(
	1,3\right) $-Bertrand curve and $\beta ^{\left( 4\right) }$ is a
	quaternionic $\left( 1,3\right) $-Bertrand mate of $\alpha ^{\left( 4\right)
	}$ then the curvature functions of $\beta ^{\left( 4\right) }$ is
	independent of the torsion $-r$ of $\alpha ^{\left( 4\right) }.$
\end{remark}

\begin{corollary}
	\label{cor 4.1 copy(1)}Let $\alpha ^{\left( 4\right) }:I\subset \mathbb{R}%
	\rightarrow $ $\mathbb{E}^{4}$ be a quaternionic $\left( 1,3\right) $%
	-Bertrand curve and $\beta ^{\left( 4\right) }$ be a quaternionic $\left(
	1,3\right) $-Bertrand mate of $\alpha ^{\left( 4\right) }.$ Then the
	curvature functions of the curve $\beta $\ which is a spatial quaternionic
	curve associated with $\beta ^{\left( 4\right) }$ are defined by%
	\begin{eqnarray*}
		\bar{k}(\bar{s}) &=&\frac{c\left[ \left( 1+d^{2}\right) \left( K-k\right)
			\left( s\right) -\left( 1+c^{2}\right) K(s)\right] }{\epsilon \delta \left(
			1+c^{2}\right) \sqrt{1+d^{2}}\left( 1-aK(s)\right) }, \\
		\bar{r}(\bar{s}) &=&-\frac{c\left \vert \left( c\left( 1+d^{2}\right) \left(
			K-k\right) \left( s\right) -\left( 1+c^{2}\right) dK(s)\right) \right \vert 
		}{\epsilon \left( 1+c^{2}\right) \sqrt{1+d^{2}}\left( 1-aK(s)\right) }.
	\end{eqnarray*}
\end{corollary}

\begin{proof}
	It is obvious from Theorem (5).
\end{proof}


\begin{thebibliography}{99}
	\bibitem{bert} Bertrand J. M., M\'{e}moire sur la th\'{e}orie des courbes 
	\'{a} double courbure, Comptes Rendus, 15, 332-350, 1850.
	
	\bibitem{nag} Bharathi K., Nagaraj M., Quaternion valued function of a real
	Serret-Frenet formulae, Indian J. Pure Appl. Math. 18 (6) 507-511.
	
	\bibitem{koca} \c{C}etin M. Kocayi\u{g}it H., On the quaternionic\
	Smarandache curves in Euclidean 3-space, Int.J. Contemp Math Sci 8(3),
	139-150, 2013.
	
	\bibitem{er} Ersoy S., Tosun M., Timelike Bertrand curves in semi-Euclidean
	space, Int. J. Math. Stat., 14(2), 78-89, 2013.
	
	\bibitem{gok1} G\"{o}k \.{I}., Okuyucu O.Z., Kahraman F., Hac\i saliho\u{g}%
	lu H. H., On the quaternionic $B_{2}-$ slant helices in the Euclidean space $%
	\mathbb{E}^{4}.$ Adv. Appl. Clifford Algebr., 21, 707-719, 2011.
	
	\bibitem{kaz4} G\"{o}k \.{I}., Kaya Nurkan S., \.{I}larslan K., On pseudo
	null Bertrand curves in Minkowski space-time, Kyungpook Math. J. 54(4),
	685-697, 2014.
	
	\bibitem{gun} G\"{u}ng\"{o}r M. A. and Tosun M., Some characterizations of
	quaternionic rectifying curves, Differ. Geom. Dyn. Syst. 13, 89-100, 2011.
	
	\bibitem{mak} Irmak Y., Bertrand Curves and Geometric Applications in Four
	Dimensional Euclidean Space, MSc thesis, Ankara University, Institute of
	Science, 2018.
	
	\bibitem{kaz5} Kahraman Aksoyak F., G\"{o}k \.{I}., \.{I}larslan K.,
	Generalized null Bertrand curves in Minkowski space-time, An. \c{S}tiint.
	Univ. Al. I. Cuza, Ia\c{s}i, Mat. (N.S.) 60 (2), 489-502, 2014.
	
	\bibitem{ak} Kahraman Aksoyak F., A new type of quaternionic Frame in $%
	\mathbb{R}^{4},$ 16 (6), 1950084 (11 pages), 2019.
	
	\bibitem{kara} Karada\u{g} M., Sivrida\u{g} A.\.{I}., Quaternion valued
	functions of a single real variable and inclined curves, Erciyes Univ. J.
	Inst. Sci. Technol 13, 23-36,1997.
	
	\bibitem{kaz} Ke\c{c}ilio\u{g}lu O., \.{I}larslan K. , Quaternionic Bertrand
	curves in Euclidean 4-space. Bull. Math. Anal. Appl. 5 (3), 27--38, 2013.
	
	\bibitem{mat} Matsuda H. and Yorozu S., Notes on Bertrand curves. Yokohama
	Math. J. 50 (1-2), 41-58, 2003.
	
	\bibitem{on} \"{O}nder M., Quaternionic Salkowski curves and quaternionic
	similar curves, Proc. Natl. Acad. Sci. India, Sect. A Phys. Sci., 90 (3),
	447-456, 2020.
	
	\bibitem{gunay} \"{O}zt\"{u}rk G., Ki\c{s}i \.{I}., B\"{u}y\"{u}kk\"{u}t\"{u}%
	k S. , Constant ratio quaternionic curves in Euclidean spaces. Adv. Appl.
	Clifford Algebr. 27 (2), 1659--1673, 2017.
	
	\bibitem{Pears} Pears L. R., Bertrand curves in Riemannian space, J. London
	Math. Soc. 1-10 (2), 180-183, 1935.
	
	\bibitem{senyurt} \c{S}enyurt S., Cevahir C., Altun Y., On spatial
	quaternionic involute curve a new view. Adv. Appl. Clifford Algebr. 27 (2),
	1815--1824, 2017.
	
	\bibitem{kaz1} U\c{c}um A., \.{I}larslan K., Sasaki M., On (1,3)-Cartan null
	Bertrand curves in semi-Euclidean 4-space with index 2, J. Geom., 107 (3),
	579-591, 2016.
	
	\bibitem{kaz2} U\c{c}um A., Ke\c{c}ilio\u{g}lu O., \.{I}larslan K.,
	Generalized Bertrand curves with spacelike (1,3)-normal plane in Minkowski
	space-time, Turkish J. Math., 40 (3), 487-505, 2016.
	
	\bibitem{kaz3} U\c{c}um A., Ke\c{c}ilio\u{g}lu O., \.{I}larslan K.,
	Generalized Bertrand curves with timelike (1,3)-normal plane in Minkowski
	space-time, Kuwait J. Sci., 42 (3), 10-27, 2015.
	
	\bibitem{yil} Y\i ld\i z \"{O}.G., \.{I}\c{c}er \"{O}., A note on evolution
	of quaternionic curves in the Euclidean space $\mathbb{R}^{4},$ Konuralp J.
	Math., 7(2), 462-469, 2019.
	
	\bibitem{yoon1} Yoon D.W. , On the quaternionic general helices in Euclidean
	4-space, Honam Mathematical J. 34(3), 381-390, 2012.
	
	\bibitem{yoon} Yoon D.W., Dae Won, Y. Tun\c{c}er, Yilmaz, M.K. Karacan,
	Generalized Mannheim quaternionic curves in Euclidean 4-space. Appl. Math.
	Sci. (Ruse) 7, 6583--6592, 2013.
\end{thebibliography}
\end{document}